\documentclass[11pt,a4paper]{article}
\usepackage{amsmath,amssymb,amsthm}
\usepackage[utf8]{inputenc}
\usepackage[T1]{fontenc}
\usepackage[colorlinks=true,linkcolor=blue,citecolor=blue,urlcolor=blue]{hyperref}
\usepackage{geometry}
\usepackage{graphicx}
\newtheorem{theorem}{Theorem}[section]
\newtheorem{lemma}[theorem]{Lemma}

\newtheorem{claim}[theorem]{Claim}

\newtheorem{corollary}[theorem]{Corollary}

\theoremstyle{definition}

\title{Avoidability of Digraphs with Height Functions and Orientations of $C_4$
}
\author{%
  Hui Lei\thanks{School of Statistics and Data Science, LPMC and KLMDASR,
    Nankai University. {\tt hlei@nankai.edu.cn}.
   Funded by the National Natural Science Foundation of China
    (Nos.\,12371351, 12431013),
    the Natural Science Foundation of Tianjin (24JCYBJC01670),
    and the Fundamental Research Funds for the Central Universities,
    Nankai University.}
  \hspace{2mm}
  Xiaoyi Wang\thanks{School of Mathematical Sciences and LPMC,
    Nankai University. {\tt 2210725@mail.nankai.edu.cn}.}
    \hspace{2mm}
  Zhijun Xu\thanks{School of Mathematical Sciences and LPMC,
    Nankai University. {\tt zhj\_xu@mail.nankai.edu.cn}.} 
    \hspace{2mm}
  Zhenyu Yang\thanks{School of Statistics and Data Science, LPMC and KLMDASR,
    Nankai University. {\tt yangzhenyu@mail.nankai.edu.cn}.}
    }
\date{}

\begin{document}
\maketitle
\begin{abstract}
A  digraph $F$ is avoidable if, for every positive integer $k$, there exists an integer $d$ such that every digraph of minimum out-degree at least $d$ contains an $F$-free subdigraph of minimum out-degree at least $k$. We prove that no digraph admitting a height function
is avoidable, where a height function increases by one along every arc.
This answers a question of Christoph, Janzer, Petrova, and
Steiner and, together with an additional avoidance argument, determines
which orientations of $C_4$ are avoidable. 

Motivated by a further question of Christoph, Janzer, Petrova, and Steiner, we also study Eulerian avoidability, in which the host digraph is required to have equal in-degree and out-degree at every vertex. We show that no one-directed complete
bipartite digraph with nonempty parts is Eulerian-avoidable. In contrast,
the orientation of $C_4$ consisting of two directed paths of
length two with common endpoints is Eulerian-avoidable. Consequently, the
one-directed complete bipartite orientation is the only orientation of
$C_4$ that is not Eulerian-avoidable,  completing the classification of 
$C_4$-orientations.
\end{abstract}


\section{Introduction}

A fundamental conjecture in extremal graph theory, posed by Thomassen
in 1983~\cite{T}, asserts that for every $k,g\in\mathbb N$, there exists
$d=d(k,g)\in\mathbb N$ such that every graph of average degree at least
$d$ contains a subgraph of average degree at least $k$ and girth at
least $g$. The conjecture remains open in general, although Kühn and
Osthus~\cite{KO} proved it for $g=6$.
The conjecture admits an equivalent formulation in terms of avoidable
graphs. A graph $F$ is \emph{avoidable} if, for every $k\in\mathbb N$,
there exists $d_F(k)\in\mathbb N$ such that every graph of minimum degree
at least $d_F(k)$ contains an $F$-free subgraph of minimum degree at
least $k$. In this terminology, Thomassen's conjecture asserts that
every graph containing a cycle is avoidable.

For digraphs, a simple reduction to bipartite subdigraphs is not possible: Thomassen~\cite{T2} showed that there exist digraphs with arbitrarily large minimum out-degree in which every directed cycle has odd length.
Christoph, Janzer, Petrova, and Steiner~\cite{CJPS} initiated the study of the natural analog of Thomassen’s conjecture for digraphs.
A digraph $F$ is \emph{avoidable} if, for every $k\in\mathbb N$, there exists $r_F(k)\in\mathbb N$ such that every  digraph  with
minimum out-degree at least $r_F(k)$ contains an $F$-free subdigraph with minimum out-degree at least
$k$. Dellamonica, Koubek, Martin, and R\"odl~\cite{DKMR}
proved that every directed cycle is avoidable. Christoph, Janzer,
Petrova, and Steiner~\cite{CJPS} proved that all orientations of $C_3$ and $C_5$ are avoidable, whereas no oriented forest is avoidable \cite{CJPS}; this builds on the proof of the KAMAK tree conjecture \cite{C} and the results of Hons et al.
\cite{Hons} on grounded forests, and no one-directed complete bipartite
digraph is avoidable. Here a \emph{one-directed complete bipartite graph} is an orientation of a complete bipartite graph all of whose arcs are directed from one to the other color class. Throughout this paper,
digraphs are finite and loopless, with no parallel arcs; digons are allowed. An \emph{oriented graph} has no such pair. 

A \emph{height function} on a digraph $F$ is a map
$h:V(F)\to\mathbb Z$ such that
$h(v)=h(u)+1$ for every arc $u\to v$ of $F$.
All the non-avoidable digraphs identified in~\cite{CJPS} admit height
functions. This led its authors to ask whether every digraph admitting
a height function is non-avoidable~\cite[Question~6.2]{CJPS}. Our first
result answers this question affirmatively.

\begin{theorem}\label{thm:main}
Every digraph admitting a height function is non-avoidable.
More precisely, for every such digraph $F$, there exists an
integer $k\ge2$ such that, for every integer $r\ge k$, there is an
oriented graph with out-degree exactly $r$ at every vertex in which
every  subdigraph of minimum out-degree at least $k$ contains $F$.
\end{theorem}

The proof first treats complete layered digraphs, into which every fixed
digraph with a height function embeds. We construct a layered digraph
with many vertices sharing each possible out-neighborhood. A finite
family of recursive trees then forces every subdigraph of sufficiently
large minimum out-degree to contain the required layers. The trees
account simultaneously for all possible choices of retained outgoing
arcs.

We next determine the avoidable orientations of $C_4$, answering
\cite[Question~6.3]{CJPS}. To distinguish the four isomorphism types,
write their arc sets on the vertex set $\{a,b,c,d\}$ as follows:
\begin{align*}
A(\vec C_4)&=\{a\to b,b\to c,c\to d,d\to a\},\\
A(C_4^{(1,3)})&=\{a\to b,b\to c,c\to d,a\to d\},\\
A(C_4^{(2,2)})&=\{a\to b,a\to d,b\to c,d\to c\},\\
A(\vec K_{2,2})&=\{a\to b,a\to d,c\to b,c\to d\}.
\end{align*}
Both $C_4^{(2,2)}$ and $\vec K_{2,2}$
admit height functions; the other two orientations do not.

\begin{theorem}\label{thm:main1}
An orientation of $C_4$ is avoidable if and only if it admits no height
function. Equivalently, precisely $\vec C_4$ and $C_4^{(1,3)}$ are
avoidable.
\end{theorem}

The non-avoidability assertions follow from Theorem~\ref{thm:main}, and
the directed cycle case follows from~\cite{DKMR}. We prove the
avoidability of $C_4^{(1,3)}$ in Appendix~\ref{app:shortcut} by adapting the argument of~\cite[Lemma~4.5]{CJPS}.

A digraph  is \emph{Eulerian} if
each vertex has equal in-degree and out-degree. A digraph $F$ is \emph{Eulerian-avoidable} if, for every
$k\in\mathbb N$, there exists $d_F(k)\in\mathbb N$ such that
every  Eulerian digraph  with
minimum out-degree at least $d_F(k)$ contains an $F$-free subdigraph with minimum out-degree at least
$k$. Every avoidable digraph is Eulerian-avoidable, but the
converse fails, as our next results show.

For positive integers $s,t$, let $\vec K_{s,t}$ denote the orientation
of $K_{s,t}$ in which every arc goes from the part of size $s$ to the
part of size $t$.

\begin{theorem}\label{thm:main3}
A one-directed complete bipartite graph is not
Eulerian-avoidable. More precisely, if $s,t\ge2$, then for every integer
$r\ge t$ there is an Eulerian oriented graph with
minimum out-degree exactly $r$ such that every  subdigraph of
minimum out-degree at least $t$ contains $\vec K_{s,t}$.
\end{theorem}

\begin{theorem}\label{thm:main2}
For every $k\in\mathbb N$, every  Eulerian digraph of minimum
out-degree at least $18k^4$ contains a spanning $C_4^{(2,2)}$-free subdigraph
with out-degree exactly $k$ at every vertex. In particular, $C_4^{(2,2)}$ is
Eulerian-avoidable.
\end{theorem}

Theorem~\ref{thm:main3} uses a cyclic sequence of large-girth bipartite
graphs and a final complete bipartite interface. The latter bounds the
number of vertices that an avoiding subdigraph can retain in the last
layer, whereas the large-girth interfaces force these numbers to
decrease when the layers are read backwards. The proof of
Theorem~\ref{thm:main2} instead uses independent choices of outgoing
arcs and the asymmetric Lov\'asz local lemma. The Eulerian condition
bounds the total probability of the forbidden copies involving any
one random choice, even when the degrees are unbounded above.

\begin{corollary}\label{cor:main3}
An orientation of $C_4$ is Eulerian-avoidable if and only if it is not
isomorphic to $\vec K_{2,2}$.
\end{corollary}

This gives an answer to~\cite[Question~6.4]{CJPS}. The complete
classification is recorded in Table~\ref{tab:c4}.
\begin{table}[htbp]
\centering
\begin{tabular}{lccc}
\hline
Orientation & Height function & Avoidable & Eulerian-avoidable\\
\hline
$\vec C_4$ & No & Yes & Yes\\
$C_4^{(1,3)}$ & No & Yes & Yes\\
$C_4^{(2,2)}$ & Yes & No & Yes\\
$\vec K_{2,2}$ & Yes & No & No\\
\hline
\end{tabular}
\caption{Avoidability of the four orientations of $C_4$.}
\label{tab:c4}
\end{table}

Section~\ref{sec:preliminaries} collects notation and auxiliary results.
We prove Theorems~\ref{thm:main}, \ref{thm:main3}, and \ref{thm:main2}
in Sections~\ref{height}, \ref{noteulerian}, and \ref{eulerian},
respectively. Appendix~\ref{app:shortcut} completes the proof of
Theorem~\ref{thm:main1}.

\section{Notation and preliminaries}\label{sec:preliminaries}

We  denote by $\binom{X}{q}$ the
family of all $q$-element subsets of a finite set $X$. For a digraph
$D$, let $V(D)$ and $A(D)$ be its vertex set and arc set. We write
$u\to v$ for the arc $(u,v)$. Let
\(
N_D^+(v)=\{w:v\to w\in A(D)\}\) and \(
N_D^-(v)=\{w:w\to v\in A(D)\}.
\)
The corresponding degrees are $d_D^+(v)=|N_D^+(v)|$ and
$d_D^-(v)=|N_D^-(v)|$. Minimum and maximum out-degrees are denoted by
$\delta^+(D)$ and $\Delta^+(D)$, respectively. We omit the subscript
when the ambient digraph is clear. For $X\subseteq V(D)$, put
$N_D^-(X)=\bigcup_{x\in X}N_D^-(x)$, and let $D[X]$ be the induced
subdigraph on $X$. A set of vertices is \emph{independent} if there is
no arc between any two of its vertices. A \emph{digon} is a pair of
oppositely directed arcs between two vertices.
A digraph is \emph{$r$-outregular} if every vertex has out-degree $r$. A rooted
\emph{out-tree} is a tree whose edges are directed away from its root.
Unless otherwise specified, a path in a digraph is directed. The girth
of a graph is the length of a shortest cycle, with girth
infinity for a forest. Auxiliary multigraphs will be explicitly
identified; their degrees count edge multiplicity.

\subsection{Probabilistic tools}

For a finite family of events $(E_i)_{i\in I}$, a graph on $I$ is a
\emph{dependency graph} if each $E_i$ is independent of the $\sigma$-algebra
generated by all events whose indices are neither $i$ nor adjacent to
$i$. We write $\Gamma(i)$ for the set of neighbors of $i$ in this graph.
When the events are functions of independent random variables, joining
two events whenever their sets of variables intersect gives a dependency
graph.

\begin{lemma}[Asymmetric Lov\'asz local lemma~\cite{EL,S}]
\label{lem:lll-asymmetric}
Let $(E_i)_{i\in I}$ be a finite family of events with a dependency
graph. If there exist numbers $x_i\in[0,1)$ such that
\[
\Pr(E_i)\le x_i\prod_{j\in\Gamma(i)}(1-x_j)
\qquad\text{for every }i\in I,
\]
then $\Pr(\bigcap_{i\in I}\overline{E_i})>0$.
\end{lemma}

We will also use the following standard consequence.
\begin{lemma}[Symmetric Lov\'asz local lemma~\cite{EL,S}]\label{lem:lll-symmetric}
Suppose that $\Pr(E_i)\le p$ for every $i$ and that a dependency graph
has maximum degree at most $\Delta$. If $ep(\Delta+1)\le1$, then with
positive probability none of the events occurs.
\end{lemma}

\begin{lemma}[Chernoff bounds]\label{lem:chernoff}
Let $X$ be a sum of independent Bernoulli random variables, with mean
$\mu>0$. Then
\[
\Pr\bigl(|X-\mu|>\mu/2\bigr)\le2e^{-\mu/12}.
\]
Moreover, for every real $u\ge\mu$,
\[
\Pr(X\ge u)\le\left(\frac{e\mu}{u}\right)^u.
\]
\end{lemma}

These estimates follow from
$\mathbb E e^{\lambda X}\le\exp(\mu(e^\lambda-1))$ and Markov's
inequality. Optimizing in $\lambda$ gives the usual multiplicative
upper and lower tail bounds, and the displayed forms follow by taking
relative deviation $1/2$ and upper threshold $u$, respectively. When
$\mu=0$, the corresponding upper-tail probability is zero for $u>0$.

\subsection{Coloring and large girth}

The appendix uses the following consequence of Alon's matrix partition
lemma.
\begin{theorem}[Alon~\cite{A}]\label{thm:alon-splitting}
Every digraph admits a coloring of its vertices with three colors such
that, for each vertex $v$, at least $d^+(v)/3$ of its out-neighbors have
a color different from that of $v$.
\end{theorem}

For completeness, we give the large-girth construction used in
Section~\ref{noteulerian}.
\begin{lemma}\label{lem:large-girth}
For every integer $r\ge2$ and every positive integer $g$, there is an
 $r$-regular bipartite graph of girth greater than $g$.
\end{lemma}

\begin{proof}
We describe a construction that increases the girth of any $r$-regular bipartite graph $B$. Starting from $K_{r,r}$ and
iterating this construction proves the lemma.

Let $Q=(\mathbb Z/2\mathbb Z)^{E(B)}$, with basis vectors
$\varepsilon_e$ indexed by $E(B)$. Form a graph $\widetilde B$ on
$V(B)\times Q$ by replacing each edge $e=uv$, with $u$ in the first
bipartition class, by the edges
\[
(u,z)(v,z+\varepsilon_e)\qquad(z\in Q).
\]
The graph $\widetilde B$ is finite, simple, bipartite, and $r$-regular.
The projection onto $B$ maps the incident edges at each lifted vertex
bijectively onto those at its image.

Let $h$ be the girth of $B$. If $\widetilde B$ had a cycle of length
at most $h$, its projection would be a closed walk of the same length
without immediate backtracking, including at the junction of its last
and first edges. A shortest repeated-vertex segment in such a walk is
a cycle, so the girth condition implies that the projected walk is a
simple cycle $C$ of length $h$. Traversing its edges changes the second
coordinate by
\[
\sum_{e\in E(C)}\varepsilon_e\ne0,
\]
since these edges are distinct. The lifted walk therefore cannot be
closed, a contradiction. Thus the girth strictly increases. Each graph
in the iteration has a cycle because it is finite with minimum degree
at least two, so finitely many iterations yield girth greater than $g$.
\end{proof}

\section{Proof of Theorem \ref{thm:main}}\label{height}

For integers $k\ge2$ and $t\ge1$, let $J_{k,t}$ be the digraph with
pairwise disjoint classes $A_0,\ldots,A_t$, each of size $k$, and all
arcs from $A_i$ to $A_{i+1}$ for $0\le i<t$, with no other arcs. Every
finite digraph admitting a height function is a subdigraph of some
$J_{k,t}$. We  prove the following construction lemma.

\begin{lemma}\label{prop:layered}
For all integers $k\ge2$, $t\ge1$, and $r\ge k$, there is a finite
$r$-outregular oriented graph $D=D(r,k,t)$ such that every nonempty
subdigraph $H\subseteq D$ with $\delta^+(H)\ge k$ contains $J_{k,t}$.
\end{lemma}

\subsection{Layers and profiles}

Fix $k,t,r$ as in Lemma \ref{prop:layered} and put
\[
p=\binom{r}{k},\qquad q=(r-1)p^{r-1}+1.
\]
Construct a digraph $D'$ with independent layers $L_0,\ldots,L_t$,
starting with $|L_t|=r$. For $i=t-1,t-2,\ldots,0$ and for each
$S\in\binom{L_{i+1}}r$, introduce $q$ distinct vertices in $L_i$ whose
out-neighborhood is exactly $S$. All vertices introduced in this way
are distinct, and there are no other arcs. Since $q\ge r$, every layer
has at least $r$ vertices. Each
vertex of $W=L_0\cup\cdots\cup L_{t-1}$ has out-degree $r$ in $D'$.

A \emph{profile} $\varphi$ assigns to each $v\in W$ a set
\[
\varphi(v)\in\binom{N_{D'}^+(v)}k.
\]
Let $\Omega$ be the finite nonempty set of all profiles.  The
following simultaneous pigeonhole argument is the key step.

\begin{claim}\label{lem:profiles}
For every nonempty $U\subseteq\Omega$ with $|U|\le r-1$, there are
$r$-element sets $S_i(U)\subseteq L_i$, for $0\le i\le t$, such that
for every $0\le i<t$:
\begin{enumerate}
\item $N_{D'}^+(v)=S_{i+1}(U)$ for every $v\in S_i(U)$;
\item each profile $\varphi\in U$ is constant on $S_i(U)$.
\end{enumerate}
\end{claim}

\begin{proof}
Set $S_t(U)=L_t$ and choose the remaining sets in decreasing order of
their indices. Write $U=\{\varphi_1,\ldots,\varphi_\ell\}$, where
$1\le\ell\le r-1$. Once $S_{i+1}(U)$ has been chosen, there are $q$ vertices
in $L_i$ with this out-neighborhood. The signatures
\[
\bigl(\varphi_1(v),\ldots,\varphi_{\ell}(v)\bigr)
\]
of these vertices take at most $p^{\ell}\le p^{r-1}$ values,  because each coordinate is a $k$-element subset
 of the same $r$-element set $S_{i+1}(U)$. Since
$q>(r-1)p^{r-1}$, at least $r$ vertices have the same signature.
Choose any $r$ of them as $S_i(U)$.
\end{proof}

For each eligible $U$, fix a sequence supplied by
Claim \ref{lem:profiles}. The choices for different sets  $U$  need not satisfy any compatibility condition.

\subsection{Recursive trees and the host digraph}

For every nonempty $U\subseteq\Omega$, define a rooted out-tree $T(U)$
whose leaves are labeled by $r$-element subsets of $L_0$.
If $|U|\le r-1$, let $T(U)$ consist of one vertex, labeled $S_0(U)$.
If $|U|\ge r$, choose distinct profiles
$\varphi_1,\ldots,\varphi_r\in U$, take disjoint copies of
\[
T(U\setminus\{\varphi_1\}),\ldots,
T(U\setminus\{\varphi_r\}),
\]
and  add a new root with an arc to the root of each copy. Fix all choices in this
recursion. The parameter sets strictly decrease in size and remain 
nonempty, so every tree is finite. Each internal vertex has exactly $r$
children. Distinct copies have disjoint vertex sets even when their parameter
sets or leaf labels coincide.

Take $D'$ and $r$ disjoint copies of $T(\Omega)$, also disjoint from $D'$,
and let $R$ be the set of their roots. Retain all arcs of these digraphs, add arcs
from each leaf to  every vertex in its label, and add all arcs from
$L_t$ to $R$. Denote the resulting digraph by $D$.

Every vertex of $D$ has out-degree exactly $r$: with arcs from internal tree vertices to children, leaves to labels, $W$ to $D'$-out-neighbors, and $L_t$ to $R$. It is loopless, parallel-arc-free, and digon-free (as $L_0\cap L_t=\emptyset$ for $t\ge1$). Hence $D$ is $r$-outregular and oriented.

\subsection{Forcing the layered digraph}

\begin{proof}[Proof of Lemma~\ref{prop:layered}]
Suppose that a  $J_{k,t}$-free subdigraph $H\subseteq D$ has
$\delta^+(H)\ge k$. For each $v\in W\cap V(H)$, choose a $k$-element
set $\varphi(v)\subseteq N_H^+(v)$; for each $v\in W\setminus V(H)$,
choose an arbitrary $k$-element subset of $N_{D'}^+(v)$. No outgoing arcs
at $W$ were added after constructing $D'$, so these choices define a
profile $\varphi\in\Omega$.

We show by induction on $|U|$ that, whenever $\varphi\in U$, the root
of every occurrence of $T(U)$ in $D$ lies outside $V(H)$. Suppose first
that $|U|\le r-1$. Its root is a leaf $a$ labeled $S_0(U)$. If
$a\in V(H)$, choose a $k$-element set
$A_0\subseteq N_H^+(a)\subseteq S_0(U)$. For $i=0,\ldots,t-1$,
let $A_{i+1}$ be the common value of $\varphi$ on $S_i(U)$, which exists 
by Claim \ref{lem:profiles}.  That claim gives $A_{i+1}\subseteq S_{i+1}(U)$, and
induction on $i$ gives
\[
A_i\subseteq S_i(U)\cap V(H),\qquad
A_{i+1}=\varphi(v)\subseteq N_H^+(v)\quad(v\in A_i).
\]
Hence all arcs from $A_i$ to $A_{i+1}$ belong to $H$. The $k$-element
sets $A_0,\ldots,A_t$ lie in distinct layers and form a copy of
$J_{k,t}$, a contradiction.

Now let $|U|\ge r$.  At least $r-1$ of the $r$ child parameter sets $U\setminus\{\varphi_j\}$ contain
 the fixed profile 
$\varphi$. By induction, their roots lie outside $V(H)$. The root of
$T(U)$ therefore has at most one possible out-neighbor in $H$ and
cannot belong to $H$, since $k\ge2$.

Taking $U=\Omega$ excludes all vertices of $R$. Since vertices of
$L_t$ have all their out-neighbors in $R$, they too are absent from
$H$. Successively, the same holds for $L_{t-1},\ldots,L_0$. Any
remaining vertex of $H$ would lie in one of the finite out-trees.
A remaining vertex of greatest depth would have no out-neighbor in
$H$, again a contradiction.
\end{proof}

\begin{proof}[Proof of Theorem~\ref{thm:main}]
The empty target is contained in every digraph and is trivially
non-avoidable. Let $F$ be nonempty and fix a height function $h$.
After adding a constant to $h$, assume that its minimum is zero. Set
\[
t=\max\{1,\max_{v\in V(F)}h(v)\},\qquad
V_i=\{v\in V(F):h(v)=i\}\quad(0\le i\le t),
\]
and let $k=\max\{2,|V_0|,\ldots,|V_t|\}$. Every arc of $F$ goes
from $V_i$ to $V_{i+1}$ for some $i<t$, so embedding each $V_i$ into
the corresponding class of $J_{k,t}$ embeds $F$ as a subdigraph.
Empty height classes and isolated vertices cause no difficulty because
the copy need not be induced.

For every $r\ge k$, Lemma~\ref{prop:layered} supplies an
$r$-outregular oriented graph $D_r$ in which every nonempty subdigraph
of minimum out-degree at least $k$ contains $J_{k,t}$ and hence $F$.
If $F$ were avoidable, choosing $r\ge\max\{k,r_F(k)\}$ would contradict
the definition of $r_F(k)$.
\end{proof}

\section{Proof of Theorem~\ref{thm:main3}}\label{noteulerian}

First suppose that $s,t\ge2$, and fix an integer $r\ge t$. Put
\[
p=(s-1)\binom{r}{t},\qquad q=p+1.
\]
By Lemma~\ref{lem:large-girth}, choose an $r$-regular
bipartite graph $B$ of girth greater than $2p$. Its bipartition classes
have the same size, say $\ell$, by regularity, and $\ell\ge r$.

Take pairwise disjoint sets $S,A_1,\ldots,A_q$ with $|S|=r$ and
$|A_i|=\ell$. Construct $D_r$ by including all arcs from $S$ to $A_1$,
a copy of $B$ directed from $A_i$ to $A_{i+1}$ for each $1\le i<q$,
and all arcs from $A_q$ to $S$. There are no other arcs. Every vertex
in $S$ has in-degree and out-degree $\ell$, and every other vertex has
both degrees $r$. Since $q\ge2$, this is an Eulerian oriented graph
with minimum out-degree exactly $r$.

Suppose that a $\vec K_{s,t}$-free subdigraph $H\subseteq D_r$
has $\delta^+(H)\ge t$. Write $n_i=|A_i\cap V(H)|$. Following outgoing
arcs in $H$ visits the parts in the cyclic order
$S,A_1,\ldots,A_q,S$, so each part meets $V(H)$ and each $n_i$ is
positive.

For every $x\in A_q\cap V(H)$, select a $t$-element subset of
$N_H^+(x)\subseteq S\cap V(H)$. No such subset can be selected by
$s$ different vertices, since they would form $\vec K_{s,t}$.
Consequently,
\begin{equation}\label{eq:last-layer}
n_q\le(s-1)\binom{|S\cap V(H)|}{t}\le p.
\end{equation}

If $n_{i+1}\le p$, the undirected graph formed by the arcs of $H$
between $A_i\cap V(H)$ and $A_{i+1}\cap V(H)$ is a forest; otherwise, we would obtain a cycle
of length at most $2n_{i+1}\le2p$, contrary to the girth of
$B$. If this forest has $e_i$ edges, then
\[
tn_i\le e_i\le n_i+n_{i+1}-1.
\]
The first inequality uses the minimum out-degree of $H$,  and the second uses the forest
bound. Since $t\ge2$ and $n_i>0$,
\begin{equation}\label{eq:layer-decrease}
n_i\le(t-1)n_i\le n_{i+1}-1\le p-1.
\end{equation}
Starting from~\eqref{eq:last-layer} and applying
\eqref{eq:layer-decrease} backwards gives
\[
1\le n_1\le n_q-(q-1)\le p-p=0,
\]
a contradiction. This proves the asserted construction. As $r$ can
be arbitrarily large while $s,t$ are fixed, no Eulerian-avoidability
threshold can exist for the target minimum out-degree $t$.

It remains to consider a part of size one. Every digraph of minimum
out-degree at least $t$ contains $\vec K_{1,t}$. Also, if a 
digraph $H$ has $\delta^+(H)\ge s$, then
\[
\sum_{v\in V(H)}d_H^-(v)=\sum_{v\in V(H)}d_H^+(v)\ge s|V(H)|,
\]
so some vertex has at least $s$ in-neighbors, giving a copy of
$\vec K_{s,1}$. Eulerian digraphs of arbitrarily large minimum
out-degree exist, for example by replacing every edge of a complete
graph by a digon. These observations settle the remaining cases.
\qed

\section{Proof of Theorem~\ref{thm:main2}}\label{eulerian}

Recall that $C_4^{(2,2)}$ has arcs $a\to b$, $a\to d$, $b\to c$, and $d\to c$.
We retain exactly $k$ outgoing arcs at each vertex independently. The
main estimate bounds the sum of the probabilities of forbidden copies
using any one of these random choices.

For $k=1$, retaining one outgoing arc at each vertex gives a spanning
$C_4^{(2,2)}$-free subdigraph. Assume that $k\ge2$, and let $D$ be a nonempty
Eulerian digraph. Write
\[
n_v=d_D^+(v)=d_D^-(v),\qquad
\delta=\min_{v\in V(D)}n_v\ge18k^4.
\]
Independently for each vertex $v$, choose a uniformly random
$k$-element subset of $N_D^+(v)$ and retain precisely the arcs from
$v$ to this subset. The resulting spanning subdigraph $H$ is
$k$-outregular.

Index the copies of $C_4^{(2,2)}$ in $D$ by their source $a$, unordered pair of
intermediate vertices $\{b,d\}$, and sink $c$, with all four vertices
distinct. For each copy, let $E$ be the event that all its arcs are
retained, and put $p_E=\Pr(E)$. Then
\begin{equation}\label{eq:diamond-probability}
p_E=\frac{\binom{n_a-2}{k-2}}{\binom{n_a}{k}}\frac{k}{n_b}\frac{k}{n_d}=\frac{k(k-1)}{n_a(n_a-1)}\frac{k}{n_b}\frac{k}{n_d}
\le\frac{k^4}{n_a^2n_bn_d}.
\end{equation}
 where the last inequality uses $n_a\ge k$. The event depends only on the random choices at
$a,b,d$, not on the choice at the sink $c$.

For vertices $x,y$, put $q(x,y)=|N_D^+(x)\cap N_D^+(y)|$. Fix a
vertex $v$.  The sum of the probabilities of the events with source $v$ is
at most
\begin{align}
\sum_{E:\,v\text{ is the source}}p_E
&\le\frac{k^4}{n_v^2}
\sum_{\{b,d\}\in\binom{N_D^+(v)}2}\frac{q(b,d)}{n_bn_d}\notag\\
&\le\frac{k^4}{n_v^2}\binom{n_v}{2}\frac1\delta
\le\frac{k^4}{2\delta}.
\label{eq:source-load}
\end{align}
Here we used $q(b,d)\le\min\{n_b,n_d\}$, hence
$q(b,d)/(n_bn_d)\le1/\max\{n_b,n_d\}\le1/\delta$, and allowing
inadmissible sinks only increases the bound.

If $v$ is an intermediate vertex and $w$ is the other intermediate
vertex, then~\eqref{eq:diamond-probability} gives
\begin{align}
\sum_{E:\,v\text{ is intermediate}}p_E
&\le\frac{k^4}{n_v}
\sum_{a\in N_D^-(v)}\frac1{n_a^2}
\sum_{w\in N_D^+(a)\setminus\{v\}}\frac{q(v,w)}{n_w}\notag\\
&\le\frac{k^4}{n_v}\sum_{a\in N_D^-(v)}\frac1{n_a}
\le\frac{k^4d_D^-(v)}{n_v\delta}
=\frac{k^4}{\delta}.
\label{eq:intermediate-load}
\end{align}
The second inequality uses $q(v,w)\le n_w$. The Eulerian hypothesis
is used in the final equality. Thus, by
\eqref{eq:source-load} and~\eqref{eq:intermediate-load},
\begin{equation}\label{eq:vertex-load}
\sum_{E:\,E\text{ uses the choice at }v}p_E
\le\frac{3k^4}{2\delta}
\qquad\text{for every }v\in V(D).
\end{equation}

Join two events if their sets of random-choice vertices intersect.
This gives a dependency graph because the choices at distinct
vertices are independent. Every event uses three vertices. Thus, by~\eqref{eq:vertex-load}, its
neighborhood $\Gamma(E)$ satisfies
\[
\sum_{E'\in\Gamma(E)}p_{E'}
\le3\cdot\frac{3k^4}{2\delta}\le\frac14.
\]
Set $x_E=2p_E$. By~\eqref{eq:diamond-probability},
$p_E\le k^4/\delta^4$, so $0\le x_E<1$. The elementary inequality
$\prod_i(1-y_i)\ge1-\sum_i y_i$ for $y_i\in[0,1]$ yields
\[
\prod_{E'\in\Gamma(E)}(1-x_{E'})
\ge1-\sum_{E'\in\Gamma(E)}x_{E'}\ge\frac12.
\]
Consequently,
\[
p_E\le x_E\prod_{E'\in\Gamma(E)}(1-x_{E'}).
\]
Lemma~\ref{lem:lll-asymmetric} implies that with positive probability
none of the events occurs. The resulting $H$ is a spanning
$k$-outregular $C_4^{(2,2)}$-free subdigraph, as required.
\qed

\appendix
\section{The remaining avoidable orientation of the four-cycle}\label{app:shortcut}

Let $Q=C_4^{(1,3)}$ be the oriented graph on $\{a,b,c,d\}$ with
arcs $a\to b$, $b\to c$, $c\to d$, and $a\to d$.
We adapt the argument of Christoph, Janzer, Petrova, and Steiner
\cite[Lemma~4.5]{CJPS}, supplying the reduction and the probabilistic
estimates needed here.

\begin{theorem}\label{thm:c413-avoidable}
The oriented graph $C_4^{(1,3)}$ is avoidable.
\end{theorem}

A tripartition $(A,B,C)$ of an oriented graph $D$ is \emph{1-typed}
if $A$, $B$, and $C$ are independent sets and every vertex has all its
out-neighbors in a single class. When $\delta^+(D)>0$, write $T_A$, $T_B$, and $T_C$ for the sets of vertices whose out-neighbors lie in
$A$, $B$, and $C$, respectively. Note that $T_A$ is independent. In fact, an arc between two of its vertices would have its head in both $A$ and $T_A$, whereas
$A\cap T_A=\varnothing$.

\begin{lemma}\label{lem:typed-reduction}
For every positive integer $q$, every digraph of minimum out-degree at
least $12q$ contains a spanning oriented subdigraph of minimum out-degree
at least $q$ admitting a 1-typed tripartition.
\end{lemma}

\begin{proof}
First eliminate digons. Form an undirected graph whose edges are the
pairs of vertices that form digons. Every undirected graph has an
orientation in which every vertex $v$ has out-degree at least
$\lfloor d(v)/2\rfloor$: add a new vertex adjacent to every odd-degree vertex, orient the edges along an Euler tour in each nontrivial component of the
resulting even graph, and delete the added vertex. Keep the indicated
arc from each digon, together with every arc not belonging to a digon.
A vertex originally incident with $t$ digons and $s$ other outgoing
arcs retains at least
\[
 s+\lfloor t/2\rfloor\ge \lfloor(s+t)/2\rfloor
\]
outgoing arcs. Therefore, the resulting oriented graph has minimum
out-degree at least $6q$.

Apply Theorem~\ref{thm:alon-splitting}: every vertex has at
least one third of its out-neighbors in other color classes. Delete
the arcs within the classes. Each vertex still has at least $2q$
out-neighbors, distributed between the other two classes. For each
vertex, keep its outgoing arcs to a class containing at least $q$ of
these neighbors. This gives the required subdigraph and tripartition.
\end{proof}

\begin{lemma}\label{lem:one-source-class}
Let $k\ge100$ be an integer. Suppose that $D$ is an oriented graph of minimum
out-degree at least $k^{20}$ with a 1-typed tripartition $(A,B,C)$.
Then $D$ contains a spanning subdigraph $D'$ in which every vertex has
out-degree $k$ and no copy of $Q$ has its source in $T_A$.
\end{lemma}

\begin{proof}
Put $d=k^{20}$. If $T_A=\varnothing$, simply retain $k$
outgoing arcs at each vertex. Henceforth assume $T_A\ne\varnothing$.
For a spanning subdigraph $J\subseteq D$, define an auxiliary
multidigraph $\mathcal H(J)$ on $T_A$ as follows. For every directed
path $u\to w\to v$ in $J$ with $u,v\in T_A$ and
$N_J^+(u)\cap N_J^+(v)\ne\varnothing$, insert an arc $u\to v$
labeled by $w$. Different choices of $w$ give parallel arcs.

The multidigraph $\mathcal H(J)$ is loopless because $D$ has no
digons. Moreover, $\mathcal H(J)$ has an arc if and only if $J$
contains a copy of $Q$ with source in $T_A$. Indeed, such an arc gives
$u\to w\to v\to x$ and $u\to x$ for a common out-neighbor $x$.
Here $u,v\in T_A$ and $w,x\in A$, and $w\ne x$ because otherwise
$v,w$ form a digon. Conversely, in a copy
$u\to w\to v\to x$, $u\to x$ with $u\in T_A$, both $w$ and $x$
lie in $A$, and the arc $v\to x$ implies $v\in T_A$. It therefore suffices to construct a spanning $k$-outregular
subdigraph $D'\subseteq D$ such that $\mathcal H(D')$
has no arcs.

First, we construct a spanning subdigraph $F\subseteq D$ and an
ordering $v_1,\ldots,v_n$ of $T_A$ such that
\begin{equation}\label{eq:appendix-sparse-predecessors} 
\begin{aligned} 
d_F^+(v)&\ge3k^4 &&(v\in T_A),\\ 
d_F^+(v)&=k &&(v\notin T_A),\\ 
\bigl|N_{\mathcal H(F)}^-(v_i)       
\cap\{v_1,\ldots,v_{i-1}\}\bigr|&\le2k &&(1\le i\le n). 
\end{aligned}
\end{equation}
The last bound counts distinct in-neighbors, and parallel arcs are counted only once.

Retain exactly $d$ outgoing arcs at each vertex of $T_A$ and exactly
$k$ at every other vertex, obtaining $F_0$. Since every out-neighbor
of a vertex of $T_A$ lies outside $T_A$, each vertex has out-degree at
most $dk$ in $\mathcal H(F_0)$, with multiplicity. Consequently, every
induced submultidigraph of $\mathcal H(F_0)$ has a vertex of total
degree at most $2dk$. Successively removing such a vertex and
reversing the removal order gives an ordering $v_1,\ldots,v_n$ in
which at most $2dk$ arcs enter $v_i$ from preceding vertices, counted
with multiplicity.

Independently retain every arc of $F_0$ with tail in $T_A$ with
probability $p=k^{-15}$, leaving all other arcs unchanged. Denote the
resulting graph by $F$. For each $i$, let $\mathcal A_i$ be the event
\[
 d_F^+(v_i)\notin[dp/2,3dp/2],
\]
and let $\mathcal B_i$ be the event that $v_i$ has more than $2k$
distinct preceding in-neighbors in $\mathcal H(F)$.
Since $dp/2=k^5/2\ge3k^4$, avoiding these events establishes
\eqref{eq:appendix-sparse-predecessors}.

By the first estimate in Lemma~\ref{lem:chernoff},
\begin{equation}\label{eq:appendix-degree-tail}
 \Pr(\mathcal A_i)\le2e^{-k^5/12}\le k^{-k}.
\end{equation}

To estimate the probability of $\mathcal B_i$, let $S_i=N_F^+(v_i)$ be the random retained out-neighborhood of $v_i$. Fix a possible value $S_i=S$ with $dp/2\le |S|\le3dp/2$, so that $\mathcal A_i$ does not occur.
This fixes all random choices at $v_i$.
For each $j<i$, let $m_{ji}$ be the number of arcs from $v_j$ to
$v_i$ in $\mathcal H(F_0)$, and let $X_j$ indicate that $v_j$ is an
in-neighbor of $v_i$ in $\mathcal H(F)$. Conditional on $S_i=S$ , \(\mathcal B_i\) is the event that $\sum_{j<i}X_j>2k$. The event $X_j=1$ requires both a retained arc from $v_j$ to $S$ and a retained arc
$v_j\to w$ for one of the $m_{ji}$ vertices $w$ labeling such an
auxiliary arc. Their respective probabilities are at most $|S|p$
and $m_{ji}p$. The two sets of candidate arcs are disjoint: a
labeling vertex $w$ satisfies $w\to v_i$, so $w\notin S$ because
$D$ has no digons. 
The two events depend on disjoint sets of independent arc choices at
$v_j$, so they are independent conditional on $S_i=S$. 
Hence 
\[
 \Pr(X_j=1\mid S_i=S)\le\frac32dp^3m_{ji}.
\]
After the choices at $v_i$ have been fixed, the variables $X_j$ are
independent, since each depends only on choices at its own vertex
$v_j$. The ordering therefore gives
\[
 \mathbb E\left[\sum_{j<i}X_j\,\middle|\,S_i=S\right]
 \le\frac32dp^3\sum_{j<i}m_{ji}
 \le3d^2kp^3=3k^{-4}\le1.
\]
The second estimate in Lemma~\ref{lem:chernoff} yields, uniformly over
the fixed choices,
\[
 \Pr(\mathcal B_i\mid S_i=S)
 \le\left(\frac{e}{2k}\right)^{2k}=k^{-k}\left(\frac{e^2}{4k}\right)^k \le k^{-k}.
\]
Together with \eqref{eq:appendix-degree-tail}, this implies
\begin{equation}\label{eq:appendix-bad-event}
\begin{aligned}
\Pr(\mathcal A_i\cup\mathcal B_i)
&=\Pr(\mathcal A_i)+\Pr(\mathcal B_i\cap\overline{\mathcal A_i})\\
  &=\Pr(\mathcal A_i)+\mathbb E\!\left[\mathbf1_{\overline{\mathcal A_i}}\Pr(\mathcal B_i\mid S_i)\right]\\
  &\le k^{-k}+k^{-k}\Pr(\overline{\mathcal A_i})\\
&\le2k^{-k}.
\end{aligned}
\end{equation}

Consider the dependency bound which can be read directly from the independent blocks of choices indexed by vertices of $T_A$.
The event $\mathcal A_i\cup\mathcal B_i$ depends only on the block
at $v_i$ and those at its preceding in-neighbors in
$\mathcal H(F_0)$, at most $2dk+1$ blocks altogether. Each block
occurs in at most $dk+1$ of these events, because
$\Delta^+(\mathcal H(F_0))\le dk$. Thus each event is independent
of all but at most
\[
 (2dk+1)(dk+1)-1\le5d^2k^2=5k^{42}
\]
others. Lemma~\ref{lem:lll-symmetric} applies, since
\[
 e\,(2k^{-k})(5k^{42}+1)
 \le12e\,k^{42-k}<1\quad(k\ge100).
\]
Therefore we can fix $F$ satisfying
\eqref{eq:appendix-sparse-predecessors}.

It remains to choose $k$ out-neighbors for each vertex of $T_A$.
Process $v_1,\ldots,v_n$ in order, leaving the outgoing arcs of
vertices outside $T_A$ unchanged. Suppose the choices at the preceding
vertices have been made. From $N_F^+(v_i)$, delete the chosen
out-neighborhoods of its preceding in-neighbors in $\mathcal H(F)$.
At most $2k^2$ candidate vertices are excluded, leaving a set $W$ of size at
least $3k^4-2k^2$. Choosing the out-neighbors of $v_i$ from $W$ ensures that no arc from a
preceding vertex to $v_i$ remains in the final auxiliary graph.

On $W$, form an undirected graph $R$ by joining distinct $w,x$ if
there is a directed path $w\to v_j\to x$ or $x\to v_j\to w$
through a preceding vertex $v_j$, using its already chosen outgoing
arcs. Each $w\in W\subseteq A$ has exactly $k$ outgoing arcs in
$F$, and each preceding $v_j$ retains exactly $k$. Hence at most
$k^2$ such paths start at $w$, and $R$ has average degree at most
$2k^2$. Every graph of average degree at most $\ell$ has an independent
set of size at least $|V|/(\ell+1)$: in a uniformly random vertex order,
the vertices preceding all their neighbors form an independent set
of expected size $\sum_v1/(d(v)+1)\ge |V|/(\ell+1)$.
Consequently, $R$ has an independent set of size at least
\[
 \frac{3k^4-2k^2}{2k^2+1}\ge k.
\]
Choose $k$ of its vertices as the out-neighbors of $v_i$.
An auxiliary arc $v_i\to v_j$ to a preceding vertex would give a
path $v_i\to w\to v_j$ and a common out-neighbor $x$ of $v_i,v_j$;
then $w,x$ would be adjacent in $R$, a contradiction. Thus there
are no auxiliary arcs in either direction between $v_i$ and any
preceding vertex. Subsequent deletions cannot create auxiliary arcs.
After all vertices have been processed, the resulting $D'$ has
out-degree exactly $k$ at every vertex and $\mathcal H(D')$ is
edgeless. This proves the lemma.
\end{proof}

\begin{proof}[Proof of Theorem~\ref{thm:c413-avoidable}]
Fix a positive integer $k$, and put $K=\max\{k,100\}$.
Let $D$ have minimum out-degree at least
\[
 d(k)=12K^{20^3}.
\]
By Lemma~\ref{lem:typed-reduction}, $D$ contains a spanning oriented subdigraph $D_0$ with a 1-typed tripartition $(A,B,C)$ and minimum
out-degree at least $K^{20^3}$. Apply
Lemma~\ref{lem:one-source-class} three times, with parameters
$K^{20^2}$, $K^{20}$, and $K$, and with the designated receiving
classes $A$, $B$, and $C$, respectively. This produces spanning
subdigraphs
\[
 D_0\supseteq D_1\supseteq D_2\supseteq D_3.
\]
The digraphs $D_1,D_2,D_3$ are $K^{20^2}$-outregular,
$K^{20}$-outregular, and $K$-outregular, respectively. The sets $T_A,T_B,T_C$ remain fixed,
since every vertex retains positive out-degree throughout.
The graph $D_1$ has no copy of $Q$ with source in $T_A$;
$D_2$ has none with source in $T_B$; and $D_3$ has none with source
in $T_C$. Deleting arcs preserves the earlier exclusions, and the
three source classes partition the vertex set. Hence $D_3$ is
$Q$-free and has minimum out-degree $K\ge k$.
\end{proof}
\end{document}